\documentclass[12pt]{article}

\usepackage{amssymb}%
\usepackage{amsmath}%
\usepackage{amsthm}%
\usepackage{mathrsfs}
\usepackage{xcolor}%

\allowdisplaybreaks%

{\theoremstyle{plain}%
 \newtheorem{theorem}{Theorem}

}
{\theoremstyle{remark}

}
{\theoremstyle{definition}

}

\begin{document}

\begin{center}
{\large $q$-analogues of $\pi$-formulas due to Ramanujan and Guillera}

 \ 

{\textsc{John M. Campbell}} 

 \ 

\emph{Dedicated to my hero Jes\'{u}s Guillera and his mathematical legacy}

\end{center}

\begin{abstract}
 The first known $q$-analogues for any of the $17$ formulas for $\frac{1}{\pi}$ due to Ramanujan were introduced in 2018 by Guo and 
 Liu (J.\ Difference Equ.\ Appl.\ 29:505--513, 2018), via the $q$-Wilf--Zeilberger method. Through a ``normalization'' method, which 
 we refer to as {\tt EKHAD}-norma-lization, based on the $q$-polynomial coefficients involved in first-order difference equations 
 obtained from the $q$-version of Zeilberger's algorithm, we introduce $q$-WZ pairs that extend WZ pairs introduced by Guillera (Adv.\ 
 in Appl.\ Math.\ 29:599--603, 2002) (Ramanujan J.\ 11:41--48, 2006). We apply our {\tt EKHAD}-normalization method to prove four 
 new $q$-analogues for three of Ramanujan's formulas for $\frac{1}{\pi}$ along with $q$-analogues of Guillera's first two series for 
 $\frac{1}{\pi^2}$. Our normalization method does not seem to have been previously considered in any equivalent way in relation to 
 $q$-series, and this is substantiated through our survey on previously known $q$-analogues of Ramanujan-type series for 
 $\frac{1}{\pi}$ and of Guillera's series for $\frac{1}{\pi^2}$. We conclude by showing how our method can be adapted to further 
 extend Guillera's WZ pairs by introducing hypergeometric expansions for $\frac{1}{\pi^2}$. 
\end{abstract}


\vspace{0.1in}


\vspace{0.1in}

\section{Introduction}
 The Wilf--Zeilberger (WZ) method~\cite{PetkovsekWilfZeilberger1996} is of great importance in combinatorics, special functions theory, 
 and computer science. One of the most notable developments in the history of WZ theory is given by the discoveries due to Guillera, 
 with a 2002 contribution of Guillera~\cite{Guillera2002} having led to a breakthrough in the areas of mathematics related to Ramanujan's 
 series for $\frac{1}{\pi}$~\cite[pp.\ 352--354]{Berndt1994PartIV}~\cite{Ramanujan1914}. Guillera applied the WZ method, 
 through the use of Zeilberger's {\tt EKHAD} package for the Maple computer algebra system, to introduce and prove the first 
 Ramanujan-type series for $\frac{1}{\pi^2}$. This motivates the development of techniques to build upon Guillera's applications of the 
 WZ method and the {\tt EKHAD} package. This provides the main purpose of our paper. 

 A discrete function $A(n ,k)$ is said to be \emph{hypergeometric} if $\frac{A(n+1, k)}{A(n, k)}$ and $\frac{A(n, k + 
 1)}{A(n, k)}$ are rational functions. A pair $(F, G)$ of bivariate, hypergeometric functions $F = F(n, k)$ and $G = G(n, k)$ is called a 
 \emph{WZ pair} if the discrete difference equation 
\begin{equation}\label{WZdifference}
 F(n+1, k) - F(n, k) = G(n, k + 1) - G(n, k) 
\end{equation}
 holds. For a WZ pair $(F, G)$, there exists a rational function $R(n ,k)$ such that $G(n, k) = R(n, k) F(n, k)$~\cite{WilfZeilberger1990}, 
 and $R(n, k)$ is referred to as the \emph{certificate} associated with $(F, G)$. Setting 
\begin{equation}\label{definitionH}
 H(n ,k) := F(n+1,n+k) + G(n,n+k), 
\end{equation}
 the relation 
\begin{equation}\label{mainGH}
 \sum_{n=0}^{\infty} G(n, 0) = \sum_{n=0}^{\infty} H(n, 0) 
\end{equation}
 holds under the assumption that both sums in \eqref{mainGH} converge and under the assumption that $\lim_{n \to \infty} \sum_{k 
 = 0}^{n - 1} F(n, k)$ vanishes~\cite{Zeilberger1993}. The relation in \eqref{mainGH} provides a key in Guillera's derivations of 
 remarkable formulas such as 
\begin{equation}\label{firstGuillera}
 \frac{128}{\pi ^2} = \sum_{n=0}^{\infty} \frac{ (-1)^{n} \binom{2n}{n}^{5} }{2^{20n}} 
 \big( 820 n^2 +180 n + 13 \big). 
\end{equation}
 The WZ pair $(F, G)$ employed by Guillera to prove \eqref{firstGuillera} was determined through the {\tt EKHAD} package and is 
 such that 
\begin{equation}\label{introGuilleraF}
 F(n, k) = 512 \frac{ (-1)^{k} }{ 2^{16n} 2^{4k} } \frac{n^3}{4n-2k-1} 
 \frac{ \binom{2n}{n}^{4} \binom{2k}{k}^{3} \binom{4n-2k}{2n-k} }{ \binom{2n}{k} \binom{n+k}{n}^{2} } 
\end{equation} 
 and $$ G(n, k) = \frac{(4 n - 2 k - 1) \left(84 k n+10 k+120 n^2+34 n+3\right)}{512 n^3} F(n, k). $$
 The problem of determining $q$-analogues of the WZ pairs applied by Guil-lera~\cite{Guillera2002} has led us to formulate an ``{\tt 
 EKHAD}-type'' method, which we refer to 
 as {\tt EKHAD}-normalization, 
 toward extending known WZ pairs. We apply our method
 to obtain four new $q$-analogues of three of Ramanujan's formulas for $\frac{1}{\pi}$
 and new $q$-analogues 
 for the first two known Ramanujan-type formulas 
 for $\frac{1}{\pi^2}$. 
 Devising $q$-analogues for $\pi$-formulas is 
 considered to be challenging~\cite{Wei2023}, which motivates the $q$-analogues
 introduced in this paper. 
 One of these $q$-analogues is highlighted as a motivating result in Section~\ref{subsectionmotivating}. Prior to this, we review some required preliminaries, as in Section~\ref{sectionPrelim}. 

  In Section~\ref{sectionsurvey}, we provide a survey on past    formulas that provide $q$-analogues of Ramanujan-type formulas for  
  $\frac{1}{\pi}$  and Guillera's formulas for $\frac{1}{\pi^2}$.   This survey is followed by an outline, in Section~\ref{EKHADoutline}, as to  
   our {\tt EKHAD}-normalization method.  This, in turn, is followed by    a review, in Section~\ref{sectionCompare}, of past methods  
    employed in the derivation of the $q$-analogues covered in Section~\ref{sectionsurvey},    and this is followed by a review of the method 
  of WZ seeds  recently due to Au~\cite{Au2025}. 
 Our main results are given in Sections~\ref{sectionqRamanujan}--\ref{sectionqgenerators}.
 We conclude in Section~\ref{sectionConclusion} by showing how our 
 {\tt EKHAD}-normalization method can be adapted
 to obtain new hypergeometric series for $\frac{1}{\pi^2}$ in the spirit of Guillera's discoveries. 

\subsection{Preliminaries}\label{sectionPrelim}
 The $\Gamma$-function is defined for arguments with a positive real part so that $\Gamma(x) = \int_{0}^{\infty} t^{x-1} e^{-t} \, dt$ 
 and its analytic continuation is obtained through the relation $\Gamma(x) = \frac{\Gamma(x+1)}{x}$. The Pochhammer symbol may 
 then be defined so that $(x)_{n} = \frac{\Gamma(x+n)}{\Gamma(x)}$, and we write 
\begin{equation}\label{ChuordinaryPochhammer}
 \left[ \begin{matrix} \alpha, \beta, \ldots, \gamma \vspace{1mm} \\ 
 A, B, \ldots, C \end{matrix} \right]_{n} = \frac{ (\alpha)_{n} (\beta)_{n} 
 \cdots (\gamma)_{n} }{ (A)_{n} (B)_{n} \cdots (C)_{n}}. 
\end{equation}
 A generalized hypergeometric series may then be defined so that 
\begin{equation}\label{definitionpFq}
 {}_{p}F_{q}\!\!\left[ \begin{matrix} 
 a_{1}, a_{2}, \ldots, a_{p} \vspace{1mm}\\ 
 b_{1}, b_{2}, \ldots, b_{q} \end{matrix} \ \Bigg| \ x 
 \right] = \sum_{n=0}^{\infty} 
 \frac{ \left( a_{1} \right)_{n} \left( a_{2} \right)_{n} 
 \cdots \left( a_{p} \right)_{n} }{ \left( b_{1} \right)_{n} 
 \left( b_{2} \right)_{n} \cdots \left( b_{q} \right)_{n} } 
 \frac{x^{n}}{n!}, 
\end{equation}
 and we define the \emph{convergence rate} of the ${}_{p}F_{q}$-series in \eqref{definitionpFq} as $x$. 
 A \emph{Ramanujan-type 
 series} is a series satisfying 
\begin{equation*}
 \sum_{n = 
 0}^{\infty} 
 \left[ \begin{matrix} \frac{1}{s}, \frac{1}{2}, 1 - \frac{1}{s} 
 \vspace{1mm} \\ 
 1, 1, 1 \end{matrix} \right]_{n} z^n (a + b n) = \frac{1}{\pi}, 
\end{equation*}
 for $s \in \{ 2, 3, 4, 6 \}$ and such that the parameters $a$, $b$, and $z$ are real algebraic numbers. 

 The {$q$-Pochhammer symbol} may be defined so that 
\begin{equation}\label{finiteqPochhammer}
 (a;q)_{n} = \prod_{k=0}^{n-1} (1- aq^{k}). 
\end{equation}
 We then let 
 $$ (a;q)_{\infty} = \prod_{k=0}^{\infty} \left( 1 - a q^{k} \right),$$
 and this allows us to extend the definition in \eqref{finiteqPochhammer}
 so as to allow for non-integer values for $n$, according to the relation 
 $$ (a;q)_{n} = \frac{ \left( a;q \right)_{\infty} }{ \left( a q^{n}; q \right)_{\infty}}, $$ 
 and we write 
 \begin{equation}\label{ChuqPochhammer}
 \left[ \begin{matrix} \alpha, \beta, \ldots, \gamma \vspace{1mm} \\ 
 A, B, \ldots, C \end{matrix} \, \Bigg| \, q \right]_{n} 
 = \frac{ \left( \alpha; q \right)_{n} 
 \left( \beta; q \right)_{n} \cdots \left( \gamma; q \right)_{n} }{ 
 \left( A; q \right)_{n} 
 \left( B; q \right)_{n} \cdots \left( C; q \right)_{n}}. 
\end{equation}
 A fundamental property concerning the $q$-Pochhammer symbol is such that 
\begin{equation}\label{limitqPochhammer}
 \lim_{q \to 1} \frac{ \left( q^{x}; q \right)_{n} }{ (1-q)^{n} } = (x)_{n}. 
\end{equation}

 We let the {$q$-bracket symbol} be such that $[n]_{q} = [n] = \frac{1-q^n}{1-q}$. 
 Also, the $q$-Gamma function may be defined so that 
 $$ \Gamma_{q}(x) = (1-q)^{1-x} \frac{ (q;q)_{\infty} }{ (q^x;q)_{\infty} }, $$ with 
\begin{equation}\label{limitqGammaGamma}
 \lim_{q\to 1^{-}} \Gamma_{q}(x) = \Gamma(x), 
\end{equation}
 and, by analogy with 
 \eqref{ChuordinaryPochhammer} and \eqref{ChuqPochhammer}, we write 
\begin{equation*}
 \Gamma_{q} \left[ \begin{matrix} \alpha, \beta, \ldots, \gamma \vspace{1mm} \\ 
 A, B, \ldots, C \end{matrix} \right] 
 = \frac{ \Gamma_{q}(\alpha) \Gamma_{q}(\beta) \cdots \Gamma_{q}(\gamma) }{ 
 \Gamma_{q}(A) \Gamma_{q}(B) \cdots \Gamma_{q}(C) }. 
 \end{equation*}
 \emph{Unilateral basic hypergeometric series} may be defined so that 
\begin{multline*}
 {}_{j}\phi_{k}\!\!\left[ \begin{matrix} 
 a_{1}, a_{2}, \ldots, a_{j} \vspace{1mm}\\ 
 b_{1}, b_{2}, \ldots, b_{k} \end{matrix} \ \Bigg| \ q; z 
 \right] \\ 
 = \sum_{n=0}^{\infty} \frac{z^n}{ (q;q)_{n} } 
 \left[ \begin{matrix} 
 a_{1}, a_{2}, \ldots, a_{j} \vspace{1mm} \\ 
 b_{1}, b_{2}, \ldots, b_{k} \end{matrix} \, \Bigg| \, q \right]_{n} 
 \left( (-1)^{n} q^{\binom{n}{2}} \right)^{1 + k - j}. 
\end{multline*}

\subsection{A motivating result}\label{subsectionmotivating}
 One of the main results introduced in this paper is the $q$-identity 
\begin{multline*}
 \sum_{n=0}^{\infty} (-1)^n 
 q^{n^2} \frac{ (1-q^{ -1 + 2n}) (2 q^{1+4n} - q^{1+2n} - 1 ) }{1 + q^{1+2n}} \\ 
 \times \left[ \begin{matrix} \frac{1}{q}, q, q, q, q \vspace{1mm} \\ 
 q^2, q^2, q^2, q^{1-2n}, q^{2+2n} \end{matrix} \, \Bigg| \, q^2 \right]_{n}
 = 
 \frac{(1-q)^2 (1+q)}{q} \left[ \begin{matrix} q, q \vspace{1mm} \\ 
 q^2, q^3 \end{matrix} \, \Bigg| \, q^2 \right]_{\frac{1}{2}}
\end{multline*}
 that we have obtained through our {\tt EKHAD}-normalization method described in Section~\ref{EKHADoutline} below. 
 By dividing both sides of the above equality 
 by $(1-q)^2$ and by then letting $q$ approach $1$, 
 and by then applying the limiting relation in \eqref{limitqPochhammer}, 
 we obtain the series 
 $$ -\frac{1}{2} \sum_{n=0}^{\infty} (-1)^{n} (2 n-1) (6 n+1)
 \left[ \begin{matrix} -\frac{1}{2}, \frac{1}{2}, \frac{1}{2}, \frac{1}{2}, \frac{1}{2}
 \vspace{1mm} \\ 
 1, 1, 1, \frac{1}{2} - n, 1 + n \end{matrix} \right]_{n}, $$
 and this may be rewritten as 
\begin{equation}\label{halfRamanujan}
 \frac{1}{2} \sum_{n=0}^{\infty} \left( \frac{1}{4} \right)^{n} 
 \left[ \begin{matrix} \frac{1}{2}, \frac{1}{2}, \frac{1}{2}
 \vspace{1mm} \\ 
 1, 1, 1 \end{matrix} \right]_{n} (6n+1), 
\end{equation}
 and, as shown by Ramanujan~\cite[pp.\ 352--354]{Berndt1994PartIV}~\cite{Ramanujan1914} in an equivalent way, 
 the hypergeometric series in \eqref{halfRamanujan}
 reduces to $\frac{2}{\pi}$. We obtain the same evaluation by dividing
 the right-hand side of the above $q$-identity by $(1-q)^2$
 and by again letting $q$ approach $1$, 
 according to the limiting property of the $q$-Gamma function shown in \eqref{limitqGammaGamma}. 

 We prove the above $q$-series identity in 
 Section~\ref{sectionqgenerators}, and our main results
 are summarized in Table~\ref{tableorg}. 

\begin{table}[h!]
	\centering
	\begin{tabular}{|c c c c|} 
		\hline
 	 Constant & Convergence rate & Guillera's WZ pair & New $q$-analogue \\ [0.5ex] 
 $\frac{1}{\pi}$ & $-\frac{1}{4}$ & \cite{Guillera2002} & Theorem~\ref{qnegquarttheorem} \\ [1.5ex] 
 $\frac{1}{\pi}$ & $\frac{1}{64}$ & \cite{Guillera2002} & Theorem~\ref{q64theorem} \\ [1.5ex] 
 $\frac{1}{\pi^2}$ & $\frac{1}{16}$ & \cite{Guillera2002} & Theorem~\ref{qGuillerafirst} \\ [1.5ex]
 $\frac{1}{\pi^2}$ & $-\frac{1}{1024}$ & \cite{Guillera2002} & Theorem~\ref{qGuillerasecond} \\ [1.5ex] 
 $\frac{1}{\pi}$ & $\frac{1}{4}$ & \cite{Guillera2006} & Theorem~\ref{Theoremposquart1} \\ [1.5ex]
	 	 $\frac{1}{\pi}$ & 
 $\frac{1}{4}$ & \cite{Guillera2006} & Theorem~\ref{Theoremposquart2} \\ [1.5ex] 
		\hline 
	\end{tabular}
	\caption{Organization of the main results. }
	\label{tableorg}
\end{table}

\section{Comparison with previous methods}
 The below survey, 
 the below summary of our {\tt EKHAD}-normalization method, 
 and the below review of previously known methods in the derivation of $q$-analogues of 
 Ramanujan-type formulas for $\frac{1}{\pi}$ and Guillera's formulas for $\frac{1}{\pi^2}$ 
 provide a demonstration of the innovative nature of 
 the {\tt EKHAD}-normalization method we introduce. 

\subsection{Survey}\label{sectionsurvey}
 As noted by Guo and Zudilin~\cite{GuoZudilin2018} the first known $q$-analogues for any of the $17$ formulas for $\frac{1}{\pi}$ 
 due to Ramanujan were introduced by Guo and Liu in 2018~\cite{GuoLiu2018}. 
 The corresponding research contribution by Guo and Liu~\cite{GuoLiu2018} 
 may thus be seen as seminal in areas of mathematics related to $q$-analogues for $\pi$-formulas, 
 and this motivates our survey and our comparison with our new normalization method. 

 Guo and Liu~\cite{GuoLiu2018} 
 introduced and proved $q$-analogues for the Ramanujan-type formulas for $\frac{1}{\pi}$ 
\begin{equation}\label{eqRamanujanposquart}
 \sum_{n = 0}^{\infty} \left( \frac{1}{4} \right)^{n} 
 \left[ \begin{matrix} \frac{1}{2}, \frac{1}{2}, \frac{1}{2} \vspace{1mm} \\ 
 1, 1, 1 \end{matrix} \right]_{n} (6n+1) = \frac{4}{\pi}
\end{equation}
 and 
\begin{equation}\label{Ramanujantypenegeighth}
 \sum_{n = 0}^{\infty} \left( -\frac{1}{8} \right)^{n} 
 \left[ \begin{matrix} \frac{1}{2}, \frac{1}{2}, \frac{1}{2} \vspace{1mm} \\ 
 1, 1, 1 \end{matrix} \right]_{n} (6n+1) = \frac{2\sqrt{2}}{\pi}. 
\end{equation}
 The Guo--Liu $q$-analogues of \eqref{eqRamanujanposquart} 
 and \eqref{Ramanujantypenegeighth} are 
\begin{equation}\label{firstGuoLiuq}
 \sum_{n=0}^{\infty} 
 q^{n^2} [6n+1] 
 \frac{ (q; q^2 )_{n}^{2} (q^2;q^4)_{n} }{ (q^4;q^4)_{n}^{3} } = 
 (1 + q) \left[ \begin{matrix} q^2, q^6 \vspace{1mm} \\ 
 q^4, q^4 \end{matrix} \, \Bigg| \, q^4 \right]_{\infty} 
\end{equation} 
 and 
\begin{equation}\label{GuoLiusecond}
 \sum_{n=0}^{\infty} 
 (-1)^n q^{3n^2} [6n+1] 
 \frac{ ( q; q^2 )_{n}^{3} }{ \left( q^{4}; q^{4} \right)_{n}^{3} 	 }
 = \left[ \begin{matrix} q^3, q^5 \vspace{1mm} \\ 
 q^4, q^4 \end{matrix} \, \Bigg| \, q^4 \right]_{\infty}, 
\end{equation}
 respectively. 

 What is regarded as the simplest out of Ramanujan's formulas for $\frac{1}{\pi}$ is 
\begin{equation}\label{simplestRamanujan}
 \sum_{n = 
 0}^{\infty} \left( -1 \right)^{n} 
 \left[ \begin{matrix} \frac{1}{2}, \frac{1}{2}, \frac{1}{2} \vspace{1mm} \\ 
 1, 1, 1 \end{matrix} \right]_{n} (4n+1) = \frac{2}{\pi}
\end{equation}
 and was included in Ramanujan's first letter to Hardy~\cite[pp.\ 23--24]{Berndt1994ParII}~\cite[p.\ xxvi]{Ramanujan1962} 
 and was originally introduced and proved by Bauer via a Fourier--Legendre expansion~\cite{Bauer1859}. 
 A notable follow-up to the research contribution of 
 Guo and Liu~\cite{GuoLiu2018} is due to Guo and Zudilin~\cite{GuoZudilin2018}, 
 who noted that a classical basic hypergeometric series identity 
 known as Jackson's formula can be used to provide 
 the $q$-analogue 
 $$ \sum_{n=0}^{\infty} (-1)^n q^{n^2} [4n+1] 
 \frac{ (q;q^2)_{n}^{3} }{ (q^2;q^2)_{n}^{3} } 
 = \left[ \begin{matrix} q, q^3 \vspace{1mm} \\ 
 q^2, q^2 \end{matrix} \, \Bigg| \, q^2 \right]_{\infty} $$ 
 of \eqref{simplestRamanujan}, and who introduced and proved 
 a $q$-analogue of Ramanujan's formula 
 \begin{equation}\label{Ramanujannegquart}
 \sum_{n = 0}^{\infty} \left( -\frac{1}{4} \right)^{n} 
 \left[ \begin{matrix} \frac{1}{4}, \frac{1}{2}, \frac{3}{4} \vspace{1mm} \\ 
 1, 1, 1 \end{matrix} \right]_{n} (20 n + 3) = \frac{8}{\pi}, 
\end{equation}
 namely 
\begin{multline*}
 \sum_{n=0}^{\infty} 
 (-1)^n q^{2n^2} 
 \frac{ \big( q^2;q^4 \big)_{n}^{2} (q;q^2)_{2n} }{ \big( q^4;q^4 \big)_{n}^2 
 (q^4;q^4)_{2n} } \Bigg( [8n+1] \\ 
 + [4n+1] \frac{q^{4n+1}}{1+q^{4n+2}} \Bigg) = (1+q) 
 \left[ \begin{matrix} q^2, q^6 \vspace{1mm} \\ 
 q^4, q^4 \end{matrix} \, \Bigg| \, q^4 \right]_{\infty}, 
\end{multline*}
 together with a $q$-analogue of Ramanujan's formula 
 \begin{equation}\label{Ramanujan1over9}
 \sum_{n = 
 0}^{\infty} \left( \frac{1}{9} \right)^{n} 
 \left[ \begin{matrix} \frac{1}{4}, \frac{1}{2}, \frac{3}{4} \vspace{1mm} \\ 
 1, 1, 1 \end{matrix} \right]_{n} (8 n + 1) = \frac{2\sqrt{3}}{\pi}, 
\end{equation}
 namely 
 $$ \sum_{n=0}^{\infty} q^{2n^2} 
 \frac{ (q;q^2)_{n}^{2} (q;q^2)_{2n} }{(q^6;q^6)_{n}^2 (q^2;q^2)_{2n} } 
 [8n+1] = \left[ \begin{matrix} q^3 \vspace{1mm} \\ 
 q^2 \end{matrix} \, \Bigg| \, q^2 \right]_{\infty} 
 \left[ \begin{matrix} q^3 \vspace{1mm} \\ 
 q^6 \end{matrix} \, \Bigg| \, q^6 \right]_{\infty}. $$

 Subsequent to the work of 
 Guo and Zudilin~\cite{GuoZudilin2018}, 
 Guillera~\cite{Guillera2018} proved a $q$-analogue of the Ramanujan-type formula 
 \begin{equation}\label{Ramanujantypeneg512}
 \sum_{n = 0}^{\infty} 
 \left( -\frac{27}{512} \right)^{n} 
 \left[ \begin{matrix} \frac{1}{6}, \frac{1}{2}, \frac{5}{6} \vspace{1mm} \\ 
 1, 1, 1 \end{matrix} \right]_{n} (154 n + 15) = \frac{32\sqrt{2}}{\pi}, 
\end{equation}
 namely 
\begin{multline*}
 \sum_{n=0}^{\infty} (-1)^n q^{7n^2} 
 \frac{ (q;q^2)_{3n} (q;q^2)_{n}^2 }{ (q^4;q^4)_{2n}^{2} (q^4;q^4)_{n} } \Bigg( [10n+1] \\ 
 + \frac{ [6n+1] }{ [4n+4] } \frac{ 
 [6n+3] q^{14n + 7} - [10n+7] q^{10n+3} }{ (1+q^{2n+1})^{2} (1 + q^{4n+2})^{2} } \Bigg) = 
 \left[ \begin{matrix} q^3, q^5 \vspace{1mm} \\ 
 q^4, q^4 \end{matrix} \, \Bigg| \, q^4 \right]_{\infty}, 
\end{multline*}
 together with a $q$-analogue of the Ramanujan series 
 \begin{equation*}
 \sum_{n = 
 0}^{\infty} 
 \left( -\frac{1}{48} \right)^{n} 
 \left[ \begin{matrix} \frac{1}{4}, \frac{1}{2}, \frac{3}{4} \vspace{1mm} \\ 
 1, 1, 1 \end{matrix} \right]_{n} (28 n + 3) 
 = \frac{16\sqrt{3}}{3\pi}, 
\end{equation*}
 namely
\begin{multline*}
 \sum_{n=0}^{\infty} (-1)^{n} q^{5n^2} 
 \frac{ (q;q^2)_{2n}^2 (q;q^2)_{n} (q^3;q^6)_{n} }{ 
 (q^6;q^6)_{n}^{2} (q^2;q^2)_{4n} } \Bigg( [10n + 1] \\ 
 + \frac{q^{8n + 2} [4n+1]}{(1+q^{2n+1})(1+q^{4n+1})(1+q^{4n+2})} \Bigg) = 
 \left[ \begin{matrix} q^3 \vspace{1mm} \\ 
 q^2 \end{matrix} \, \Bigg| \, q^2 \right]_{\infty} 
 \left[ \begin{matrix} q^3 \vspace{1mm} \\ 
 q^6 \end{matrix} \, \Bigg| \, q^6 \right]_{\infty}. 
 \end{multline*}

 A corrected version of a $q$-analogue of Ramanujan's formula \eqref{eqRamanujanposquart} due to Guo 
 in 2020 (cf.~\cite[Theorem 1.1]{Guo2020}) is such that 
 $$ \sum_{n=0}^{\infty} 
 q^{2n^2} \frac{ ( q;q^2 )_{n}^4 }{ ( q^2; q^2 )_{n}^{2} (q^2;q^2)_{2n} } 
 \left( [4n+1] - \frac{q^{4n+1} [2n+1]^2}{ [4n+2] } \right) 
 = \left[ \begin{matrix} q, q^3 \vspace{1mm} \\ 
 q^2, q^2 \end{matrix} \, \Bigg| \, q^2 \right]_{\infty}. $$ 
 Chu~\cite{Chu2018} and, subsequently, Chen and Chu~\cite{ChenChu2021RJ}, obtained $q$-analogues of Rama-nujan-type
 formulas for $\frac{1}{\pi}$ and Guillera's formulas for $\frac{1}{\pi^2}$, 
 including an equivalent version of the Guo formula given above, 
 with similar $q$-analogues being given for 
 the Ramanujan formula in \eqref{Ramanujannegquart} and the Ramanujan formula 
\begin{equation}\label{Ramanujan1over64}
 \sum_{n = 0}^{\infty} 
 \left( \frac{1}{64} \right)^{n} 
 \left[ \begin{matrix} \frac{1}{2}, \frac{1}{2}, \frac{1}{2} \vspace{1mm} \\ 
 1, 1, 1 \end{matrix} \right]_{n} (42n+5) 
 = \frac{16}{\pi}. 
\end{equation}
 Chu~\cite{Chu2018} (cf.~\cite{ChenChu2021RJ}) also introduced and proved the $q$-analogue 
\begin{multline*}
 \sum_{n=0}^{\infty} (-1)^n q^{\frac{n^2}{2}} 
 \frac{ \big(q^{\frac{1}{2}}; q\big)_{n}^{5} }{ \big( -q^{\frac{1}{2}}; q^{\frac{1}{2}} \big)_{2n} 
 (q;q)_{n}^{5} } \\ 
 \times \frac{1 + q^{n + \frac{1}{2}} + q^{3n+\frac{1}{2}} + q^{4n+1} 
 - 4 q^{2n + \frac{1}{2}} }{ \big( 1-q^{\frac{1}{2}} \big) 
 \big( 1- q^{\frac{3}{2}} \big) \big( 1 + q^{n + \frac{1}{2}} \big) } = 
 \Gamma_{q} \left[ \begin{matrix} 
 1, 1, 1, 1 \vspace{1mm} \\ 
 \frac{1}{2}, \frac{1}{2}, \frac{1}{2}, \frac{5}{2} \end{matrix} \right] 
 \end{multline*}
 of the Guillera formula~\cite{Guillera2006}
\begin{equation*}
 \sum_{n = 
 0}^{\infty} 
 \left( -\frac{1}{4} \right)^{n} 
 \left[ \begin{matrix} 
 \frac{1}{2}, \frac{1}{2}, \frac{1}{2}, \frac{1}{2}, \frac{1}{2} \vspace{1mm} \\ 
 1, 1, 1, 1, 1 \end{matrix} \right]_{n} (20n^2 + 8 n + 1) 
 = \frac{8}{\pi^2}, 
\end{equation*}
 along with $q$-analogues of the Guillera formula in \eqref{firstGuillera} and the formula 
 \begin{equation}\label{Guillera16squared}
 \sum_{n=0}^{\infty} \left( \frac{1}{16} \right)^{n} 
 \left[ \begin{matrix} 
 \frac{1}{4}, \frac{1}{2}, \frac{1}{2}, \frac{1}{2}, \frac{3}{4} \vspace{1mm} \\ 
 1, 1, 1, 1, 1 \end{matrix} \right]_{n} (120 n^2+34 n+3) = \frac{32}{\pi ^2}. 
\end{equation}
 introduced and proved by Guillera~\cite{Guillera2002}. 
 Chen and Chu~\cite{ChenChu2021IJNT} have also introduced a $q$-analogue of 
 the formula 
\begin{equation*}
 \sum_{n=0}^{\infty} 
 \left( \frac{27}{64} \right)^{n} 
 \left[ \begin{matrix} 
 \frac{1}{3}, \frac{1}{2}, \frac{1}{2}, \frac{1}{2}, \frac{2}{3} \vspace{1mm} \\ 
 1, 1, 1, 1, 1 \end{matrix} \right]_{n} \left( 74n^2 + 27 n + 3 \right) 
 = \frac{48}{\pi ^2} 
 \end{equation*}
 introduced by Guillera~\cite{Guillera2011}, namely 
 \begin{multline*} 
 \frac{ \big( q^{\frac{3}{2}}, q \big)_{\infty}^{4} }{ 
 (q;q)_{\infty}^{2} (q^2;q)_{\infty}^{2} } {}_{2}\phi_{1}\!\!\left[ \begin{matrix} 
 q^{\frac{1}{2}}, q^{\frac{1}{2}} \vspace{1mm}\\ 
 0 \end{matrix} \ \Bigg| \ q; \, q 
 \right] = \sum_{n=0}^{\infty} q^{n} 
 \frac{ \big( q^{\frac{1}{2}}; q \big)_{n}^{2} 
 \big( q^{\frac{3}{2}}; q \big)_{n}^{4} 
 (q^2;q)_{3n} }{ (q^2;q)_{2n}^{3} (q;q)_{n}^{3} } \\ 
 \times \left( 1 + \frac{ (1-q^n) (2-2 q^{\frac{1}{2} + 3n} +q^n - q^{\frac{1}{2} + 2n} ) 
 (1 + q^{\frac{1}{2} +n} )^2 }{ 
 \big( 1 - q^{\frac{1}{2} +n} \big) 
 (1 - q^{1+3n}) (1 + q^n +q^{2n}) } \right). 
\end{multline*}
 The above listing of previously known $q$-analogues of
 Ramanujan-type series for $\frac{1}{\pi}$
 and for Guillera's series for $\frac{1}{\pi^2}$ is complete, to the best of our knowledge. 

\subsection{{\tt EKHAD}-normalization}\label{EKHADoutline}
 Suppose we have a WZ pair $(F, G)$ that can be applied in some specified way (for example, via the relation in \eqref{mainGH}) 
 to obtain a closed-form evaluation. 
 Moreover, suppose that it is possible to express 
 $F(n, k)$ in the form 
 $$ F(n, k) 
 = S(n) \frac{ \left( \alpha_{1}(n) \right)_{\beta_{1}(k)} 
 \left( \alpha_{2}(n) \right)_{\beta_{2}(k)} \cdots 
 \left( \alpha_{\ell_1}(n) \right)_{\beta_{\ell_1}(k)} }{ 
 \left( \gamma_{1}(n) \right)_{\delta_{1}(k)} 
 \left( \gamma_{2}(n) \right)_{\delta_{2}(k)} \cdots 
 \left( \gamma_{\ell_2}(n) \right)_{\delta_{\ell_2}(k)} } $$ 
 for a single-variable hypergeometric expression $S(n)$
 and for degree-$1$ (single-variable)
 polynomials $\alpha_{1}(n)$, $\alpha_{2}(n)$, $\ldots$, $\alpha_{\ell_1}(n)$, 
 $\beta_{1}(k)$, $\beta_{2}(k)$, $\ldots$, $\beta_{\ell_1}(k)$, 
 $\gamma_{1}(n)$, $\gamma_{2}(n)$, $\ldots$, $\gamma_{\ell_{2}}(n)$, 
 $\delta_{1}(k)$, $\delta_{2}(k)$, $\ldots$, $\delta_{\ell_{2}}(k)$ 
 with rational coefficients. As a $q$-analogue of 
 $\frac{F(n, k)}{S(n)}$, we set 
\begin{equation}\label{setFnkq}
 \mathscr{F}(n, k, q) = 
 q^{p(k)} \frac{ \left( q^{\alpha_{1}(n)}; q \right)_{\beta_{1}(k)} 
 \left( q^{\alpha_{2}(n)}; q \right)_{\beta_{2}(k)} \cdots 
 \left( q^{\alpha_{\ell_1}(n)}; q \right)_{\beta_{\ell_1}(k)} }{ 
 \left( q^{\gamma_1(n)}; q \right)_{\delta_1(k)} 
 \left( q^{\gamma_2(n)}; q \right)_{\delta_2(k)} 
 \cdots \left( q^{\gamma_{\ell_2}(n)}; q \right)_{\delta_{\ell_2}(k)}}
\end{equation}
 for a fixed power $p(k)$ that we typically set as $k$, 
 and we write $\mathscr{F}(n, k) = \mathscr{F}(n, k, q)$. 
 By applying the $q$-version of Zeilberger's algorithm (available via the {\tt EKHAD} Maple package) to 
 \eqref{setFnkq}, suppose that this produces a first-order recurrence of the form 
 $$ p_{1}(n, q) \mathscr{F}(n+1, k) + p_{2}(n, q) \mathscr{F}(n, k) 
 = \mathscr{G}(n,k+1) - \mathscr{G}(n, k) $$
 for $q$-polynomials $p_{1}(n, q)$ and $p_2(n, q)$
 and where $\mathscr{G}(n, k)$ 
 is such that $\frac{\mathscr{G}(n, k)}{ \mathscr{F}(n, k) }$
 is a $q$-rational function. We then ``normalize'' the original input function in 
 \eqref{setFnkq} by letting 
\begin{equation}\label{lettingoverlinescrF}
 \overline{\mathscr{F}}(n, k, q) = (-1)^{n} 
 \left( \prod_{i=1}^{n-1} \frac{ p_{1}(i, q) }{ p_{2}(i, q) } \right) 
 \mathscr{F}(n, k). 
\end{equation}
 An equivalent normalization operation for ordinary hypergeometric series (as opposed to $q$-series) was outlined in our recent 
 work~\cite{Campbell2025}, but this operation does not seem to have been applied previously in relation to $q$-series. 
 By applying the $q$-Zeilberger algorithm to 
 \eqref{lettingoverlinescrF}, we find that the application of the normalization operator
 to $\mathscr{F}(n, k)$ has the effect of producing the $q$-WZ difference equation
 $$ \overline{\mathscr{F}}(n+1, k) - \overline{\mathscr{F}}(n, k) 
 = \overline{\mathscr{G}}(n, k+1) - \overline{\mathscr{G}}(n, k), $$ 
 for the companion $\overline{\mathscr{G}}$ to 
 $\overline{\mathscr{F}}$ obtained via the application of the $q$-Zeilberger algorithm
 to $\overline{\mathscr{F}}$, 
 with $\frac{ \overline{\mathscr{G}}(n, k) }{ \overline{\mathscr{F}}(n, k) }$
 as a $q$-rational function. 
 Informally, by then mimicking how the WZ pair $(F, G)$ was applied 
 to derive a closed-form expression, 
 but with the use of the $q$-WZ pair $( \overline{\mathscr{F}}, \overline{\mathscr{G}} )$
 used in place of $(F, G)$, 
 this can typically be used to produce a $q$-analogue of the same closed-form formula. 
 This is clarified through our applications of this method. 

 \subsection{Comparison}\label{sectionCompare}
 As indicated above, to show the originality of our {\tt EKHAD}-normalization meth-od, we provide an overview as to the methods that 
 were applied in the derivations of the previously known $q$-analogues surveyed in Section~\ref{sectionsurvey}, together with a 
 summary of the method of WZ seeds recently introduced by Au~\cite{Au2025}. 

 The Guo--Liu $q$-series in \eqref{firstGuoLiuq}
 is such that the partial sums of this series may be rewritten so that 
\begin{multline*}
 \sum_{n=0}^{m-1} q^{n^2} [6n+1]
 \frac{ (q;q^2)_{n}^{2} (q^2;q^4)_{n} }{ (q^4;q^4)_{n}^{3} } \\ 
 = \sum_{n=1}^{m} \frac{q^{(m-n)^2} (q^2;q^4)_{m} 
 (q;q^2)_{m-n} (q;q^2)_{m+n-1} }{ (1-q) (q^4;q^4)_{m-1}^{2} (q^4;q^4)_{m-n} 
 (q^2;q^4)_{n}}, 
\end{multline*}
 and, by making the substitution $n \mapsto m - n$ on the right-hand side and then letting $m \to \infty$, and by then making the 
 substitution $q \mapsto -q$ and applying a classical $q$-series identity referred to as Slater's identity~\cite{Slater1952}, this gives us a 
 derivation of \eqref{firstGuoLiuq}, and similarly for the Guo--Liu $q$-series in \eqref{GuoLiusecond}. The truncated series identity given 
 above is proved via a $q$-WZ pair, 
 but the normalization process described in 
 Section~\ref{EKHADoutline} is 
 not involved in the work of Guo and Liu, 
 and our applications of {\tt EKHAD}-normalization do not rely 
 on truncations of the summands of our $q$-series analogues 
 or classical $q$-series identities such as Slater's identity. 

 The Guo--Liu approach outlined above was subsequently developed by Guo and Zudilin~\cite{GuoZudilin2018}, in the following manner. 
 The first Guo--Liu $q$-series shown above was derived using a $q$-WZ pair from the work of Guo~\cite{Guo2018J2} related to the 
 supercongruences of Van Hamme. This $q$-WZ pair is a $q$-analogue of a WZ pair due to Zudilin related to a 
 supercongruence, but the 
 $q$-Zeilberger algorithm and 
 the normalization process in Section~\ref{EKHADoutline} were not involved. More explicitly, by taking the 
 $q$-WZ pair $(F, G)$ employed to prove the truncated $q$-series identity given above, and by then writing $\tilde{F}(n, k) = F(n, -k)$ 
 and $\tilde{G}(n, k) = G(n, -k)$, we obtain the WZ difference equation $$ \tilde{F}(n, k+1) - \tilde{F}(n, k) = \tilde{G}(n+1, k) - \tilde{G}(n, 
 k), $$ and this was used to derive the relation 
 $$ \sum_{n=0}^{\infty} \tilde{F}(n, k) = \sum_{n=0}^{\infty} \tilde{F}(n, 0)
 = \frac{ (1 + q) (q^2;q^4)_{\infty} (q^6;q^4)_{\infty} }{ (q^4;q^4)_{\infty}^{2} }, $$ 
 which can be shown to produce the Guo--Zudilin $q$-analogue of 
 the Ramanujan formula in \eqref{Ramanujannegquart}.
 In contrast, the Guo--Zudilin $q$-analogue of 
 Ramanujan's formula \eqref{Ramanujan1over9} 
 was derived as a limiting case of a cubic transformation given in Gasper and Rahman's monograph
 on basic hypergeometric series~\cite[Equation (3.8.18)]{GasperRahman2004}. 

 Using the $q$-WZ pair $(F, G) = (F_1, G_1)$ 
 associated with Guo and Liu's proof of \eqref{GuoLiusecond}, 
 Guillera~\cite{Guillera2018} then set 
\begin{equation}\label{setF2nk}
 F_{2}(n, k) = F_{1}(2n, k - n) 
\end{equation}
 and 
\begin{equation}\label{20250901148AQMA1}
 G_{2}(n, k) = G_{1}(2n, k-n) + G_{1}(2n+1, k-n) -F_{1}(2n+2,k-n-1) 
\end{equation}
 and then obtained that 
 $$ \sum_{n=0}^{\infty} G_{2}(n, k) = \sum_{n=0}^{\infty} G_{1}(n, k) 
 = \frac{ (q^{3}; q^{4})_{\infty} (q^5; q^4)_{\infty} }{ (q^4;q^4)_{\infty}^{2} } $$
 for $k = 0, 1, \ldots$, with the $k = 0$ case yielding
 a $q$-analogue of the Ramanujan-type $\pi$-formula
 on display in \eqref{Ramanujantypeneg512}. 
 A closely related approach was subsequently applied by Guo in 2020~\cite{Guo2020} 
 to obtain a $q$-analogue of the Ramanujan formula in \eqref{eqRamanujanposquart}, 
 using a $q$-WZ pair involving 
 $$ F(n, k) = (-1)^{n+k} q^{(n-k)^{2}} [4n+1] 
 \frac{ (q;q^2)_{n}^{2} (q;q^2)_{n+k} }{ (q^2;q^2)_{n}^{2} 
 (q^2;q^2)_{n-k} (q;q^2)_{k}^2 } $$
 obtained by Guo in 2018~\cite{Guo2018central}. It was not specified
 how this $q$-WZ pair was obtained, 
 but it can be shown that {\tt EKHAD}-normalization 
 does not produce the same $q$-WZ pair. 
 Starting with a WZ pair $(F, G)$ such that 
\begin{equation}\label{20250831819PM1}
 F(n, k) = 18 (-1)^{k} \frac{ \left( \frac{1}{2} \right)_{n+k}^{2} 
 \left( \frac{1}{2} \right)_{2n - k - 1} \left( \frac{1}{2} \right)_{k} }{ 
 (1)_{n-1}^{2} (1)_{2n + 2k} } \frac{3^k}{9^n} 
\end{equation}
 given in Guillera's thesis~\cite[p.\ 28]{Guillera2007}, Guillera, through a ``trial-and-error''
 approach, obtained a $q$-WZ pair including 
\begin{equation}\label{trialanderror}
 F_{1}(n, k) 
 = \frac{(-1)^{k}}{1-q} q^{2n^2 + 4 n k - k^2} \frac{ (q;q^2)_{n+k}^2 (q;q^2)_{2n-k-1} }{ 
 (q^6;q^6)_{n-1}^{2} (q^2;q^2)_{2n+2k} } (q^3;q^6)_{k}, 
\end{equation} 
 and this was applied to produce Guillera's $q$-analogue of Ramanujan's series
 of convergence rate 	 $-\frac{1}{48}$. 
 The $q$-Zeilberger algorithm
 was not used to obtain \eqref{trialanderror} from 
 \eqref{20250831819PM1}, and the above formulation of {\tt EKHAD}-normalization is not applicable
 to \eqref{20250831819PM1}, 
 and transformations of WZ pairs as in \eqref{setF2nk}
 and \eqref{20250901148AQMA1} are not involved in our work. 

 The techniques given by Chu~\cite{Chu2018} rely on recursions 
 for Bailey's bilateral well poised series, which is defined so that 
 $$ \Omega(a; b, c, d, e) := \sum_{k=0}^{\infty} \frac{1-q^{2k}a}{1-a} 
 \left[ \begin{matrix} b, c, d, e \vspace{1mm} \\ 
 qa/b, qa/c, qa/d, qa/e \end{matrix} \, \Bigg| \, q \right]_{k} 
 \left( \frac{qa^2}{bcde} \right)^{k}, $$
 	 in conjunction with the $q$-Dougall formula 
\begin{equation*}
 \Omega(a; b, c, d, a) = 
 \left[ \begin{matrix} qa, qa/bc, qa/bd, qa/cd \vspace{1mm} \\ 
 qa/b, qa/c, qa/d, qa/bcd \end{matrix} \, \Bigg| \, q \right]_{\infty}
\end{equation*}
 for $|qa/bcd| < 1$, and a similar approach was employed in the 
 subsequent work of Chen and Chu~\cite{ChenChu2021IJNT,ChenChu2021RJ}, 
 and, in all three cases, the $q$-Zeilberger's algorithm is not involved and $q$-WZ pairs are not involved. 

\subsection{WZ seeds}
 A notable and recent development 
 in the application of WZ theory 
 is given by the method of WZ seeds due to Au~\cite{Au2025}. 
 Although this method was not applied in relation to $q$-series, 
 we briefly review this method, to again emphasize the originality of our normalization method. 

 The identity that is central to Au's method of WZ seeds 
 \cite{Au2025} may be summarized as follows. 
 For a WZ pair $(F, G)$, suppose that 
 both $\sum_{k=0}^{\infty} F(0, k)$ and $\sum_{n=0}^{\infty} G(n, 0)$ converge, 
 and suppose that 
 $\lim_{k \to \infty} G(n, k) = g(n)$ exists for each $n \in \mathbb{N}$
 and that $\sum_{n \geq 0} g(n)$ converges.
 Then $\lim_{n \to \infty} \sum_{k=0}^{\infty} F(n, k)$ exists and is finite, and, moreover, the relation 
\begin{equation}\label{displayWZseeds}
 \sum_{k=0}^{\infty} F(0, k) + \sum_{n=0}^{\infty} g(n) = \sum_{n=0}^{\infty} G(n, 0) 
 + \lim_{n \to \infty} \sum_{k = 0}^{\infty} F(n, k) 
\end{equation}
 holds. By then extracting coefficients from both sides of 
 \eqref{displayWZseeds} with respect to a given parameter, 
 this can be used, as demonstrated by Au, to produce fast converging $\pi$ formulas
 in the spirit of Ramanuajn and Guillera. 
 Zeilberger's algorithm is not involved in the method of WZ seeds 
 and our normalization procedure is not involved in Au's method, and we leave it to a separate project
 to produce $q$-series identities using WZ seeds. 

\section{New $q$-analogues of Ramanujan's formulas}\label{sectionqRamanujan}
 Out of the series expansions for $\frac{1}{\pi}$ discovered by Ramanujan, the first WZ proof for any such expansion was due to 
 Zeilberger~\cite{EkhadZeilberger1994}. In 2002, Guillera~\cite{Guillera2002} introduced WZ proofs of Ramanujan's formulas 
 \eqref{Ramanujannegquart} and \eqref{Ramanujan1over64}. 
 The $F$-entry in the WZ pair $(F, G)$ Guillera applied to prove these Ramanujan formulas is 
\begin{equation}\label{GuilleraFRamanujan}
 F(n, k) = 64 \frac{ (-1)^{n} (-1)^k }{ 2^{10n} 2^{2k} } 
 \frac{n^2}{4n-2k-1} \frac{ \binom{2k}{k}^2 \binom{2n}{n}^2 \binom{4n-2k}{2n-k} }{ \binom{2n}{k} \binom{n+k}{n} }. 
\end{equation}
 This was obtained by Guillera through the {\tt EKHAD} package. By rewriting \eqref{GuilleraFRamanujan} as 
\begin{equation}\label{rewriteGFR}
 F(n, k) = -64 \left( -\frac{1}4{} \right)^{n} \left[ \begin{matrix} 
 -\frac{1}{4}, \frac{1}{4}, \frac{1}{2} \vspace{1mm} \\ 
 1, 1, 1 \end{matrix} \right]_{n} 
 n^2 \left[ \begin{matrix} 
 \frac{1}{2}, \frac{1}{2} \vspace{1mm} \\ 
 \frac{3}{2} - 2 n, 1 + n \end{matrix} \right]_{k}, 
\end{equation}
 this has led us toward a way of extending Guillera's 
 WZ pair $(F, G)$, through the use of the {\tt EKHAD}-normalization procedure described above. 
 Experimentally, we have discovered that the $q$-analogue 
\begin{equation}\label{qofPochammerk}
 q^{k} \left[ \begin{matrix} q^{\frac{1}{2}}, q^{\frac{1}{2}} \vspace{1mm} \\ 
 q^{\frac{3}{2} - 2 n}, q^{1+n} \end{matrix} \, \Bigg| \, q \right]_{k} 
\end{equation}
 of the combination of Pochhammer symbols in \eqref{rewriteGFR}
 indexed by $k$ 
 satisfies a first-order recurrence via the $q$-version of Zeilberger's algorithm. 
 The {\tt EKHAD}-normalization of \eqref{qofPochammerk} 
 has led us to construct a $q$-WZ pair 
 to prove the following new result, 
 which (as we later clarify) 
 provides a $q$-analogue of Ramanujan's formula \eqref{Ramanujannegquart}. 

\begin{theorem}\label{qnegquarttheorem}
 For the $q$-polynomial 
 $$\rho(n, q) = -2 q^{2 n+1}+q^{4 n+1}-q^{4 n+2}+q^{6 n+1}+q^{8 n+2}+q^{10 n+3}-1,$$ 
\noindent the equality 
\begin{multline*}
 \sum_{n=0}^{\infty} (-1)^n q^{3(n-1) (n+1)} 
 \left[ \begin{matrix} q \vspace{1mm} \\ 
 -q^2, -q^2, q^2, q^2, q^2, -q^3, -q^3 \end{matrix} \, \Bigg| \, q^2 \right]_{n} \\ 
 \times \big[ \begin{matrix} q, q^3 \vspace{1mm} 
 \end{matrix} \, \big| \, q^4 \big]_{n} \, \rho(n, q) = 
 \frac{(q-1) (q+1)^3}{q^3} 
 \left[ \begin{matrix} q, q \vspace{1mm} \\ 
 q^2, q^3 \end{matrix} \, \Bigg| \, q^2 \right]_{\frac{1}{2}} 
\end{multline*}
 holds for $q$ such that $0< |q| < 1$. 
\end{theorem}

\begin{proof}
 We let 
\begin{multline*}
 F(n, k; q) := 4 (-1)^n q^{\frac{3 n^2}{2} + k -1} 
 \frac{ \big(1 + q^{\frac{1}{2}} \big)^3 \left(1-q^n\right)^2}{ \big(1 - q^{\frac{1}{2}} \big)^2} 
 \left[ \begin{matrix} q^{-\frac{1}{2}}, q^{\frac{1}{2}} 
 \end{matrix} \, \big| \, q^2 \right]_{n} \\ 
 \times \left[ \begin{matrix} q^{\frac{1}{2}} \vspace{1mm} \\ 
 -1, -1, -q^{\frac{1}{2}}, -q^{\frac{1}{2}}, q, q, q \end{matrix} \, \Bigg| \, q \right]_{n} 
 \left[ \begin{matrix} q^{\frac{1}{2}}, q^{\frac{1}{2}} \vspace{1mm} \\ 
 q^{\frac{3}{2}-2n}, q^{1+n} \end{matrix} \, \Bigg| \, q \right]_{k}. 
\end{multline*}
 We then let 
\begin{multline*}
 R(n, k; q) := \frac{q^{2 n}-q^{k+\frac{1}{2}}}{q^{k} \left(1-q^n\right)^2 \left(1 + q^n\right)^2} \\ 
 \times \frac{ q^{k+2 n+\frac{1}{2}}-2 q^{n+\frac{1}{2}}-q^{2 n+1}+q^{3
 n+\frac{1}{2}}+q^{4 n+1}+q^{5
 n+\frac{3}{2}}-1}{ 2 q^{n+1}+q^{2 n+\frac{3}{2}} + q^{\frac{1}{2}}}, 
\end{multline*}
 with $G(n, k; q) := R(n, k; q) F(n, k; q)$. 
 Let $n$ and $k$ be nonnegative real numbers, let $p$ be a real number such that $|p| < 1$, 
 and let $\eta = n + p$. 
 We may verify that the difference equation 
 $$ F(\eta + 1, k; q) - F(\eta, k; q) = G(\eta, k + 1; q) - G(\eta, k; q ) $$ 
 holds. A telescoping phenomenon 
 then gives us that 
 $$ F(m + p + 1, k; q) - F(p, k; q)
 = \sum_{n = 0}^{m} \big( G(n+p,k+1;q) - G(n+p,k;q) \big) $$ 
 for positive integers $m$. 
 We proceed to let $m$ approach $\infty$
 and to let $p$ approach $0$. 
 We then find that 
 	$\lim_{p \to 0} F(p, k; q)$ vanishes from the vanishing of 
 $F(0, k; q)$. Similarly, we find that $\lim_{m \to \infty} F(m + p, k; q)$ 
 vanishes, since, as $m$ approaches $\infty$, the expression $F(m, k; q)$ (setting $p = 0$)
 approaches 
\begin{multline*}
 \lim_{m \to \infty} \frac{ (-1)^{m+1} q^{\frac{3 (m+1)^2 }{2}} }{ \left( q^{\frac{3}{2} - 2(m+1)}; q \right)_{k} } 
 \frac{4 \left(\sqrt{q}+1\right)^3 q^{k-1}}{\left(1-\sqrt{q}\right)^2} 
 \left[ \begin{matrix} q^{\frac{1}{2}} \vspace{1mm} \\ 
 -1, -1, -q^{\frac{1}{2}}, -q^{\frac{1}{2}}, q, q, q \end{matrix} \, \Bigg| \, q \right]_{\infty} \\ 
 \times \left[ \begin{matrix} 
 q^{-\frac{1}{2}}, q^{\frac{1}{2}} \end{matrix} \, \Bigg| \, q^{2} \right]_{\infty} 
 \left[ \begin{matrix} q^{\frac{1}{2}}, q^{\frac{1}{2}} \end{matrix} \, \Bigg| \, q \right]_{k}, 
\end{multline*}
 and we thus obtain the vanishing of 
 $\lim_{m \to \infty} F(m, k; q)$ 
 from the vanishing of 
 $$ \lim_{m \to 
 \infty} \frac{ q^{\frac{3 (m + 
 1)^2 }{2}} }{ \left( q^{\frac{3}{2} - 2(m + 1)}; q \right)_{k} } = 0. $$
 From the above application of telescoping, together with 
 the vanishings we obtain 
 from the given limiting operations, we find that 
\begin{equation}\label{minfinityp0}
 \lim_{p \to 0} \sum_{n=0}^{\infty} 
 G(n+p,k;q) = \lim_{p \to 0} \sum_{n=0}^{\infty} 
 G(n+p,k+1; q). 
\end{equation}
 Carlson's theorem~\cite[\S5.3]{Bailey1935} can then be applied to show that both sides of \eqref{minfinityp0} are constant 
 for every value of $k$, including non-integers, so that 
\begin{equation}\label{consequenceCarlson1935}
 \lim_{p \to 0} \sum_{n=0}^{\infty} 
 G(n+p,0;q) = \lim_{p \to 0} \left( G\left( p, \frac{1}{2}; q \right) + \sum_{n = 1}^{\infty} 
 G\left(n+p,\frac{1}{2}; q \right) \right). 
\end{equation}

 To apply an interchange of limiting operations to
\begin{equation}\label{interchangelimsumG}
 \lim_{p \to 0} \sum_{n=1}^{\infty} G\left( n + p, \frac{1}{2}; q \right), 
\end{equation}
 we begin by rewriting the limit in \eqref{interchangelimsumG} as 
\begin{equation*}
 \lim_{\ell \to \infty} \sum_{n=1}^{\infty} \left( \Re G\left( n + \frac{1}{\ell+\frac{1}{2}}, \frac{1}{2}; q \right) 
 + i \, \Im G\left( n + \frac{1}{\ell+\frac{1}{2}}, \frac{1}{2}; q \right) \right), 
\end{equation*}
 noting that the value $\ell + \frac{1}{2}$ involved above
 is such that the above summand is defined for 
 all $\ell \in \mathbb{N}$ and for all $n \in \mathbb{N}$. 
 Depending on the value of $q$, we have that the real-valued sequence 
\begin{equation*}
 \left( \Re G\left( n + \frac{1}{\ell+\frac{1}{2}}, \frac{1}{2}; q \right) : \ell \in \mathbb{N} \right) 
\end{equation*}
 is either eventually monotonically increasing or decreasing. 
 We may assume without loss of generality that 
\begin{equation}\label{202509061014AMqqq2A}
 \left( \Re G\left( n + \frac{1}{\ell+\frac{1}{2}}, \frac{1}{2}; q \right) : \ell \in \mathbb{N}_{\geq m} 
 \right) 
\end{equation}
 is nonnegative and monotonically decreasing for some positive integer $m$, 
 as equivalent arguments can be applied in the remaining
 cases. For each $\ell \in \mathbb{N}_{\geq m}$, we set 
\begin{equation}\label{definefell}
 f_{\ell} := \Re G\left( n + \frac{1}{\ell+\frac{1}{2}}, \frac{1}{2}; q \right), 
\end{equation}
 and we let \eqref{definefell} 
 be understood as a measurable function $f_{\ell}\colon [1, \infty) \to [0, \infty)$
 with respect to the measure spaces $(1, 2, \ldots)$ and $(0, 1, \ldots)$, 
 for each positive integer $\ell$. Writing $f$ in place of 
 $ \Re G\left( n, \frac{1}{2}; q \right)$, Since 
 $f_{\ell} \to f$ almost everywhere, and since 
 $f_{m} \geq f_{m+1} \geq \cdots$, the monotone convergence theorem gives us that 
\begin{equation}\label{integralidentity}
 \lim_{\ell \to \infty} \int_{[1, \infty)} f_{\ell} = \int_{[1, \infty)} f, 
\end{equation}
 where the integrals in \eqref{integralidentity}
 are understood to be with respect to the measure space $(1, 2, \ldots)$, so that 
 $$ \lim_{\ell \to \infty} \sum_{n=1}^{\infty} \Re G\left( n + \frac{1}{\ell+\frac{1}{2}}, \frac{1}{2}; q \right) 
 = \sum_{n=1}^{\infty} \lim_{\ell \to \infty} \Re G\left( n + \frac{1}{\ell+\frac{1}{2}}, \frac{1}{2}; q \right), $$
 and a symmetric argument applies in the remaining case for the imaginary part. 

 Since $\lim_{p \to 0} G\big(n+p, \frac{1}{2}; q \big)$ vanishes for positive integers $n$, 
 the above application of the monotone convergence theorem together with 
 \eqref{consequenceCarlson1935} gives us that 
\begin{equation}\label{negquartlastdisplay}
 \lim_{p \to 0} \sum_{n=0}^{\infty} G(n + p, 0; q) = \lim_{p \to 0} G\left( p, \frac{1}{2}; q \right). 
\end{equation}
 By an equivalent application monotone convergence theorem, relative to our proof of \eqref{negquartlastdisplay}, 
 we are permitted to again interchange the limiting operations on the left of \eqref{negquartlastdisplay}. 
 After applying the substitution $q \mapsto q^2$, we obtain 
 an equivalent version of the desired result. 
\end{proof}

 Dividing both sides of the equality in Theorem~\ref{qnegquarttheorem} by $q - 1$ and by then letting $q$ approach $1$, the 
 summand, for a nonnegative integer $n$, of the left-hand side reduces to the summand of the Ramanujan series in 
 \eqref{Ramanujannegquart} for the same argument $n$, and the right-hand side reduces to $\frac{8}{\pi}$ according to the known 
 property of the $q$-Gamma function whereby 
\begin{equation}\label{closedqGamma}
 \lim_{q \to 1} \Gamma_{q^2}\left( \frac{1}{2} \right) = \sqrt{\pi}.
\end{equation}

 As clarified below, Theorem~\ref{q64theorem} provides a $q$-analogue of Ramanujan's formula \eqref{Ramanujan1over64}. 

\begin{theorem}\label{q64theorem}
 For the $q$-polynomials 
 $$ \rho_1(n, q) = 
 \big(1-q^{2 n-1}\big) \big(3 q^{10 n+3} +3 q^{8 n+2} -q^{6 n+3} +2 q^{6 n+1} -3 q^{4 n+2} -3 q^{2 n+1} -1\big) $$
 and $\rho_{2}(n, q) = \left(1 + q^{2 n+1}\right)^2$, the equality 
\begin{multline*}
 \sum_{n=0}^{\infty} (-1)^n q^{ (n-1) (3 n+5) } 
 \left[ \begin{matrix} 
 q, q, q, q, q^3 
 \vspace{1mm} \\ 
 -q^2, -q^2, q^2, q^2, q^2, q^{3-4n}, q^{2n+2} \end{matrix} \, \Bigg| \, q^2 \right]_{n} \\ 
 \times \left[ \begin{matrix} 
 q^{-1}, q 
 \vspace{1mm} \\ 
 q^2, q^6 \end{matrix} \, \Bigg| \, q^4 \right]_{n} \frac{\rho_1(n, q)}{\rho_2(n, q)} 
 = \frac{(q-1)^2 (q+1)^2}{q^6} 
 \left[ \begin{matrix} q, q \vspace{1mm} \\ 
 q^2, q^3 \end{matrix} \, \Bigg| \, q^2 \right]_{\frac{1}{2}} 
\end{multline*}
 holds for $q$ such that $0 < |q| < 1$. 
\end{theorem}

\begin{proof}
 For the same WZ pair $(F, G)$ involved in our proof of Theorem~\ref{q64theorem}, by setting $H(n, k; q)$ as in \eqref{definitionH} so that 
\begin{equation}\label{Hwithqdefinition}
 H(n , k; q) := F(n + 1, n+ k; q) + G(n, n + k; q), 
\end{equation}
 an equivalent version of Zeilberger's identity in \eqref{mainGH} gives us that 
\begin{equation}\label{GHgives}
 \lim_{p \to 0} \sum_{n=0}^{\infty} G(n+p, 0; q) = \lim_{p \to 0} \sum_{n=0}^{\infty} 
 H(n+p,0; q). 
\end{equation}
 By again applying Carlson's theorem, we obtain from \eqref{GHgives} that 
\begin{equation}\label{againCarlsonGH}
 \lim_{p \to 0} G\left( p, \frac{1}{2}; q \right) 
 = \lim_{p \to 0} \sum_{n=0}^{\infty} H(n+p, 0; q). 
\end{equation}
 Applying the monotone convergence theorem to interchange the order of limiting operations on the right of 
 \eqref{againCarlsonGH}, and applying the substitution $q \mapsto q^2$, we obtain an equivalent version of the desired result. 
\end{proof}

 Multiplying by $ q (q+1)^2/(1-q)^2$
 both sides of the $q$-identity in Theorem~\ref{q64theorem}, 
 and then letting $q$ approach $1$, 
 the summand of the infinite series for a given natural number $n$
 reduces to the summand of Ramanujan's series 
 \eqref{Ramanujan1over64} for the same integer $n$, 
 and we may again apply the $q$-Gamma relation in \eqref{closedqGamma}
 to obtain the desired multiple of $\frac{1}{\pi}$. 

\section{$q$-analogues of Guillera's formulas}
 By rewriting Guillera's $F$-function in \eqref{introGuilleraF} so that 
\begin{equation*}
 F(n, k) = -512 \left( -\frac{1}{16} \right)^{n} \left[ \begin{matrix} 
 -\frac{1}{4}, \frac{1}{4}, \frac{1}{2}, \frac{1}{2}, \frac{1}{2} \vspace{1mm} \\ 
 1, 1, 1, 1, 1 \end{matrix} \right]_{n} 
 n^3 \left[ \begin{matrix} 
 \frac{1}{2}, \frac{1}{2}, \frac{1}{2} \vspace{1mm} \\ 
 \frac{3}{2} - 2 n, 1 + n, 1 + n \end{matrix} \right]_{k}, 
\end{equation*}
 this leads us to apply {\tt EKHAD}-normalization using a $q$-analogue of the factor 
\begin{equation}\label{firstlabelqGuillera}
 \left[ \begin{matrix} 
 \frac{1}{2}, \frac{1}{2}, \frac{1}{2} \vspace{1mm} \\ 
 \frac{3}{2} - 2 n, 1 + n, 1 + n \end{matrix} \right]_{k}, 
\end{equation}
 namely 
\begin{equation}\label{EKHADqGuillera}
 q^{k} \left[ \begin{matrix} q^{\frac{1}{2}}, q^{\frac{1}{2}}, q^{\frac{1}{2}} \vspace{1mm} \\ 
 q^{\frac{3}{2} - 2 n}, q^{1+n}, q^{1+n} \end{matrix} \, \Bigg| \, q \right]_{k}. 
\end{equation}
 Using the $q$-version of Zeilberger's algorithm, 
 we find that \eqref{EKHADqGuillera} satisfies a first-order recurrence, 
 and this has led us to construct 
 a $q$-WZ pair to prove the following $q$-analogue of the Guillera formula in 
 \eqref{Guillera16squared}. 

\begin{theorem}\label{qGuillerafirst}
 For the $q$-polynomial 

 \ 

\noindent $\rho(n, q) = 
 3 q^{2 n+1}-3 q^{4 n+1}+3 q^{4 n+2}-2 q^{6 n+1}-3 q^{6 n+2}+q^{6 n+3}+q^{8 n+1}-3 q^{8 n+2}-2 q^{8 n+3}+3 q^{10 n+2}-3 q^{10 n+3}+3 q^{12
 n+3}+q^{14 n+4}+1,    $ 

 \ 

\noindent the relation 
\begin{multline*}
 \sum_{n=0}^{\infty} q^{2 n^2} 
 \left[ \begin{matrix} q, q, q \vspace{1mm} \\ 
 -q^2, -q^2, -q^2, q^2, q^2, q^2, q^2, q^2, -q^3, -q^3, -q^3 
 \end{matrix} \, \Bigg| \, q^2 \right]_{n} \\ 
 \times \big[ \begin{matrix} q, q^3 
 \end{matrix} \, \big| \, q^4 \big]_{n} \rho(n, q) = 
 (q-1)^2 (q + 1)^4 
 \left[ \begin{matrix} q, q, q \vspace{1mm} \\ 
 q^2, q^2, q^3 \end{matrix} \, \Bigg| \, q^2 \right]_{\frac{1}{2}} 
\end{multline*}
 holds for $q$ such that $|q| < 1$. 
\end{theorem}

\begin{proof}
 By setting $ F(n, k; q)$ as 
\begin{multline*}
 8 q^{k+n^2} 
 \frac{\left(q^{\frac{1}{2}}+1\right)^5 \left(1-q^n\right)^3}{\left(1-q^{\frac{1}{2}}\right)^3 q^{\frac{1}{2}}} 
 \left[ \begin{matrix} q^{\frac{1}{2}}, q^{\frac{1}{2}}, q^{\frac{1}{2}} \vspace{1mm} \\ 
 -1, -1, -1, -q^{\frac{1}{2}}, -q^{\frac{1}{2}}, -q^{\frac{1}{2}}, q, q, q, q, q 
 \end{matrix} \, \Bigg| \, q \right]_{n} \\ 
 \times \left[ \begin{matrix} q^{-\frac{1}{2}}, q^{\frac{1}{2}} 
 \end{matrix} \, \Bigg| \, q^2 \right]_{n} 
 \left[ \begin{matrix} q^{\frac{1}{2}}, q^{\frac{1}{2}}, q^{\frac{1}{2}} \vspace{1mm} \\ 
 q^{\frac{3}{2} - 2 n}, q^{1+n}, q^{1+n} \end{matrix} \, \Bigg| \, q \right]_{k}, 
\end{multline*}
 we obtain the rational certificate $R(n, k; q)$ equal to 
 $$ \frac{q^{\frac{1}{2}-k} \left(q^{k+1}-q^{2 n+\frac{1}{2}}\right)}{ q^{3/2} \left(1-q^{2 
 n}\right)^3 \left(1+q^{\frac{2 n + 1}{2} }\right)^3 } \tau(n, k) $$
 for the $q$-polynomial
\begin{multline*}
 \tau(n, k) := 
 -3 q^{k+2 n+\frac{1}{2}}-3 q^{k+3 n+1}+q^{k+4 n+\frac{1}{2}}-2 q^{k+4 n+\frac{3}{2}} + 
 3 q^{k+5 n+1} \\ 
 + 3 q^{k+6 n+\frac{3}{2}}+q^{k+7 n+2}+3 
 q^{n+\frac{1}{2}}+3 q^{2 n+1}-2 q^{3 n+\frac{1}{2}}+q^{3 n+\frac{3}{2}}-3 q^{4 n+1}-3 q^{5 n+\frac{3}{2}}+1. 
\end{multline*}
 By setting $G(n, k; q) := R(n, k; q) F(n, k; q)$, we may then following a similar approach as in our proof of 
 Theorem~\ref{qnegquarttheorem}. 
\end{proof}

 Multiplying both sides of the identity in Theorem~\ref{qGuillerafirst} by $\frac{(1+q)^2}{4 (1-q)^2 q^2}$ and letting $q$ 
 approach $1$, the summand for the series in Theorem~\ref{qGuillerafirst} for a given natural number $n$
 approaches the summand of the Guillera series in \eqref{Guillera16squared}
 for the same index $n$, 
 and we may apply the $q$-Gamma relation in \eqref{closedqGamma}
 to obtain a multple of $\frac{1}{\pi^2}$. 

 For the same $q$-WZ pair $(F, G)$ involved in our proof of Theorem~\ref{qGuillerafirst}, 
 according to the relation 
\begin{equation*}
 \lim_{p \to 0} \sum_{n=0}^{\infty} G(n+p, 0; q) = \lim_{p \to 0} \sum_{n=0}^{\infty} 
 H(n+p,0; q). 
\end{equation*}
 for $H$ as in \eqref{Hwithqdefinition}, 
 this can be used to obtain a $q$-analogue of 
 the Guillera formula 
\begin{equation*}
 \frac{128}{\pi ^2} = \sum_{n = 0}^{\infty} 
 \left( -\frac{1}{1024} \right)^{n} 
 \left[ \begin{matrix} 
 \frac{1}{2}, \frac{1}{2}, \frac{1}{2}, \frac{1}{2}, \frac{1}{2} \vspace{1mm} \\ 
 1, 1, 1, 1, 1 \end{matrix} \right]_{n} (820 n^2+180 n+13), 
\end{equation*}
 according to a similar approach as in our proof of Theorem~\ref{qGuillerafirst}. 

\begin{theorem}\label{qGuillerasecond}
 For the $q$ polynomials 

 \ 

\noindent $\rho_{1}(n, q)$ $=$ 
 $ q^{2 (n+1)^2} $ $+$ 
 $ 5 q^{2 (n+1) (n+3)}$ $ -$ 
 $q^{2 n^2+2 n+3}$ $-$ 
 $ 3 q^{2 n^2+4 n+4}$ $+$ 
 $ q^{2 n^2+6 n+2}$ $+$ 

\noindent $3 q^{2 n^2+6 n+3}$ $-$ 
 $q^{2 n^2+6 n+5}$ $-$ 
 $q^{2 n^2+8 n+1}$ $+$ 
 $ 3 q^{2 n^2+8 n+3} $ $ + $
 $ 6 q^{2 n^2+8 n+4} $ $ - $ 

\noindent $ 3 q^{2 n^2+10 n+2}$ $ - $ $ 5 q^{2 n^2+10 n+3}$ $ + $ 
 $ q^{2 n^2+10 n+4} $ $ + $ 
 $ 5 q^{2 n^2+10 n+7} $ $ - $ 
 $ 6 q^{2 n^2+12 n+3} $ $ - $ 

\noindent $ 6 q^{2 n^2+12 n+4} $ $ - $
 $ 5 q^{2 n^2+12 n+5}$ $ - $
 $ 6 q^{2 n^2+12 n+6}$ $ - $
 $ q^{2 n^2+12 n+8} $ $ + $
 $ 5 q^{2 n^2+14 n+2} $ $ + $ 

\noindent $ q^{2 n^2+14 n+3} $ $ - $
 $ 2 q^{2 n^2+14 n+5} $ $ - $ 
 $ 5 q^{2 n^2+14 n+6} $ $ - $ 
 $ 17 q^{2 n^2+14 n+7} $ $ - $
 $ 3 q^{2 n^2+14 n+9} $ $ + $
 
\noindent $ 6 q^{2 n^2+16 n+3} $ $+ $ 
 $ 3 q^{2 n^2+16 n+4} $ $+ $
 $ 6 q^{2 n^2+16 n+5}$ $+ $
 $ 17 q^{2 n^2+16 n+6} $ $ + $
 $ q^{2 n^2+16 n+7} $ $ + $ 

\noindent $ 2 q^{2 n^2+16 n+8} $ $ - $ 
 $ q^{2 n^2+16 n+10} $ $ - $
 $ q^{2 n^2+18 n+2} $ $ + $
 $ 2 q^{2 n^2+18 n+4} $ $ + $ 
 $ q^{2 n^2+18 n+5} $ $ + $

\noindent $ 17 q^{2 n^2+18 n+6} $ $ + $ 
 $ 6 q^{2 n^2+18 n+7} $ $ + $
 $ 3 q^{2 n^2+18 n+8} $ $ +$ 
 $ 6 q^{2 n^2+18 n+9} $ $ - $
 $ 3 q^{2 n^2+20 n+3} $ $ - $

\noindent $ 17 q^{2 n^2+20 n+5} $ $ - $
 $ 5 q^{2 n^2+20 n+6} $ $ - $
 $ 2 q^{2 n^2+20 n+7} $ $ + $
 $ q^{2 n^2+20 n+9} $ $ + $
 $ 5 q^{2 n^2+20 n+10} $ $ - $

\noindent $ q^{2 n^2+22 n+4} $ $ - $ 
 $ 6 q^{2 n^2+22 n+6} $ $ - $ 
 $ 5 q^{2 n^2+22 n+7} $ $ - $ 
 $ 6 q^{2 n^2+22 n+8} $ $ - $
 $ 6 q^{2 n^2+22 n+9} $ $ + $

\noindent $ 5 q^{2 n^2+24 n+5} $ $ + $
 $ q^{2 n^2+24 n+8} $ $ - $ 
 $ 5 q^{2 n^2+24 n+9} $ $ - $
 $ 3 q^{2 n^2+24 n+10} $ $ + $
 $ 5 q^{2 n^2+26 n+6} $ $ + $ 

\noindent $ 6 q^{2 n^2+26 n+8} $ $ + $
 $ 3 q^{2 n^2+26 n+9} $ $ - $
 $ q^{2 n^2+26 n+11} $ $ - $
 $ q^{2 n^2+28 n+7} $ $ + $
 $ 3 q^{2 n^2+28 n+9} $ $ + $

\noindent $ q^{2 n^2+28 n+10} $ $ - $
 $ 3 q^{2 n^2+30 n+8} $ $ + $
 $ q^{2 n^2+30 n+10} $ $ - $ 
 $ q^{2 n^2+32 n+9} $

 \ 

\noindent and $\rho_{2}(n, q) = 1-q^{4 n-1}$, the relation 
\begin{multline*}
 \sum_{n=0}^{\infty} 
 \left[ \begin{matrix} q, q, q, q, q, q \vspace{1mm} \\ 
 -q^2, -q^2, -q^2, q^2, q^2, q^2, -q^3, -q^3, -q^3, q^4, q^4, q^{3-4n}, q^{4+2n}, q^{4+2n}
 \end{matrix} \, \Bigg| \, q^2 \right]_{n} \\ 
 \times \big[ \begin{matrix} q^{-1}, q 
 \end{matrix} \, \big| \, q^4 \big]_{n} \frac{\rho_1(n, q)}{ \rho_2(n ,q) } = 
 (1-q)^3 q^2 (q+1)^4 \left(1-q^2\right)^2 
 \left[ \begin{matrix} q, q, q \vspace{1mm} \\ 
 q^2, q^2, q^3 \end{matrix} \, \Bigg| \, q^2 \right]_{\frac{1}{2}} 
\end{multline*}
 holds for $q$ such that $|q| < 1$. 
\end{theorem}

\begin{proof}
 This may be proved by analogy with the proof of Theorem~\ref{q64theorem}. 
\end{proof}

\section{$q$-generators}\label{sectionqgenerators}
 We have discovered how our {\tt EKHAD}-normalization method can    be applied to obtain further $q$-analogues derived     using WZ   
  pairs introduced  in the work of     Guillera on WZ generators for Ramanujan's formulas   for $\frac{1}{\pi}$~\cite{Guillera2006}.   
  In this regard, 
    Theorem~\ref{Theoremposquart1} below gives a 
 $q$-analogue of Ramanujan's formula for $\frac{1}{\pi}$ of convergence rate $\frac{1}{4}$, 
 and Guillera's derivation of this result according to the $j = 1$ case given by Guillera's classification system~\cite{Guillera2006}
 relies on a WZ pair involving
 $$ G(n, k) = 
 \frac{ (-1)^{n} (-1)^{k} }{2^{6n} 2^{2k}} 
 \frac{ \binom{2n}{n}^{3} \binom{2k}{k}^{2} }{2^{2k} \binom{n-1/2}{k} \binom{n+k}{n} } (4n+1). $$
 Theorem~\ref{Theoremposquart2} below 
 relies on the $j = 2$ case according to Guillera's classification system and the associated WZ pair involving
 $$ G(n, k) = \frac{(-1)^{k} }{2^{8n} 2^{2k}} \frac{ \binom{2n}{n}^{2} 
 \binom{2k}{k} \binom{2n+2k}{n+k} }{2^{2k} \binom{n-1/2}{k}} (6n+2k+1). $$

\begin{theorem}\label{Theoremposquart1}
 The $q$-series identity in Section~\ref{subsectionmotivating} holds for $|q| < 1$. 
\end{theorem}

\begin{proof}
 This can be shown through the application of {\tt EKHAD}-normalization to 
\begin{equation*}
 q^{k} \left[ \begin{matrix} q^{\frac{1}{2}}, q^{\frac{1}{2}} \vspace{1mm} \\ 
 q^{\frac{3}{2} - n}, q^{1+n} \end{matrix} \, \Bigg| \, q \right]_{k}, 
\end{equation*}
 by analogy with the preceding proofs. 
\end{proof}

\begin{theorem}\label{Theoremposquart2}
 The $q$-identity 
\begin{multline*}
 \sum_{n=0}^{\infty} 
 q^{2 n^2} \frac{ (q - q^{2n}) (2q^{1+4n} - q^{1+2n} - 1 ) }{ (1+q^{2n}) (1 + q^{1+2n}) } \\
 \times \left[ \begin{matrix} \frac{1}{q}, q, q \vspace{1mm} \\ 
 -1, -q, q^2, q^2, q^2 \end{matrix} \, \Bigg| \, q^2 \right]_{n} 
 = 
 \frac{(1-q)^2 (1+q) }{2} \left[ \begin{matrix} q, q \vspace{1mm} \\ 
 q^2, q^3 \end{matrix} \, \Bigg| \, q^2 \right]_{\frac{1}{2}}
\end{multline*}
 holds for $|q| < 1$. 
\end{theorem}

\begin{proof}
 This can be shown through the application of {\tt EKHAD}-normalization to 
\begin{equation*}
 q^{k} \left[ \begin{matrix} q^{\frac{1}{2}}, q^{\frac{1}{2}+n} \vspace{1mm} \\ 
 q^{\frac{3}{2} - n}, q^{1+n} \end{matrix} \, \Bigg| \, q \right]_{k}, 
\end{equation*}
 by analogy with the preceding proofs. 
\end{proof}

\section{Conclusion}\label{sectionConclusion}
 Our {\tt EKHAD}-normalization technique can be used to produce $q$-analogues of many further series for $\pi$, and this includes 
 $q$-analogues of further Ramanujan-type series for $\frac{1}{\pi}$ proved via the WZ method by Guillera~\cite{Guillera2006}. For the 
 sake of brevity, we leave this to a separate project, and we briefly conclude by considering how our {\tt EKHAD}-normalization 
 technique can be applied to obtain hypergeometric evaluations for $\frac{1}{\pi}$ and $\frac{1}{\pi^2}$, using variants and 
 extensions of the hypergeometric $F$-functions applied by Guillera~\cite{Guillera2002}. 

 As a generalization of the hypergeometric function in \eqref{firstlabelqGuillera}, we define
\begin{equation}\label{generalization1Conclude}
 \mathcal{F}_{\alpha}(n, k) := \left[ \begin{matrix} 
 \alpha, \frac{1}{2}, 1 - \alpha
 \vspace{1mm} \\ 
 \frac{3}{2} - 2 n, n + 1, n + 1 
 \end{matrix} \right]_{k + \beta n}. 
\end{equation}
 By applying {\tt EKHAD}-normalization, 
 for special cases of the parameters and variables involved in 
 \eqref{generalization1Conclude}, and applying WZ relations as in \eqref{minfinityp0} and 
 \eqref{GHgives}, this leads us to expansions for $\frac{1}{\pi^2}$
 and $\frac{1}{\pi}$ that motivate further applications of our method. 
 For example, 
 applying {\tt EKHAD}-normalization 
 in \eqref{generalization1Conclude} 
 and setting 
 $(\alpha, \beta, k) = \big( \frac{1}{2}, 1, 0 \big)$ 
 and applying the relation in \eqref{mainGH}, 
 we obtain that 
\begin{align*}
 \frac{2048}{\pi ^2} 
 = & \sum_{n = 
 0}^{\infty} 
 \left( \frac{1}{729} \right)^{n} 
 \left[ \begin{matrix} 
 \frac{1}{4}, \frac{1}{4}, \frac{1}{4}, \frac{1}{2}, \frac{1}{2}, \frac{1}{2}, \frac{3}{4}, \frac{3}{4}, \frac{3}{4} 
 \vspace{1mm} \\ 
 1, 1, 1, 1, 1, \frac{4}{3}, \frac{4}{3}, \frac{5}{3}, \frac{5}{3} 
 \end{matrix} \right]_{n} \\
 & 
 \times \text{ \footnotesize{$ \big( 372736 n^6+815616 n^5+682752 n^4+275168 n^3+55536 n^2+5310 n+207 \big)$}}. 
\end{align*}
 Similarly, by 
 applying {\tt EKHAD}-normalization in \eqref{generalization1Conclude}
 and setting 
 $(\alpha, \beta, k) = \big( \frac{1}{4}, 1, 0 \big)$ 
 and applying the relation in 
 \eqref{mainGH}, 
 we obtain that 
\begin{align*}
 \frac{2048}{\pi }
 = & \sum_{n = 
 0}^{\infty} 
 \left( \frac{1}{729} \right)^{n} 
 \left[ \begin{matrix} 
 \frac{1}{4}, \frac{1}{4}, \frac{1}{4}, \frac{1}{2}, \frac{3}{4}, \frac{3}{4}, \frac{3}{4} 
 \vspace{1mm} \\ 
 1, 1, 1, \frac{4}{3}, \frac{4}{3}, \frac{5}{3}, \frac{5}{3} 
 \end{matrix} \right]_{n} \\ 
 & \ \ \ \ \ \ \ \times \text{\footnotesize{$ \big( 186368 n^5+384512 n^4+293344 n^3+100760 n^2+14890 n+651 \big).$}} 
\end{align*}
 Similarly, by applying 
 {\tt EKHAD}-normalization to 
\begin{equation}\label{EKHADfor2764}
 \left[ \begin{matrix} 
 \frac{1}{2}, \frac{1}{2}, 2 n \vspace{1mm} \\ 
 n+1, n+1, 1 
 \end{matrix} \right]_{k}, 
\end{equation}
 we recover Guillera's WZ pair $(F, G)$ whereby 
 $$ F(n, k) = \left[ \begin{matrix} 
 \frac{1}{2}, \frac{1}{2}, \frac{1}{2}, 1 + \frac{k}{2}, \frac{1}{2} + \frac{k}{2} \vspace{1mm} \\ 
 1, 1, 1, k+1, k+1 
 \end{matrix} \right]_{n} 
 \left[ \begin{matrix} 
 \frac{1}{2}, \frac{1}2{} \vspace{1mm} \\ 
 1, 1 
 \end{matrix} \right]_{k} \frac{96 n^3}{2n+k} $$ 
 applied by Guillera~\cite{Guillera2011} to prove the remarkable result 
\begin{equation*}
 \frac{48}{\pi^2} 
 = \sum_{n = 
 0}^{\infty} 
 \left( \frac{27}{64} \right)^{n} 
 \left[ \begin{matrix} 
 \frac{1}{3}, \frac{1}{2}, \frac{1}{2}, \frac{1}{2}, \frac{2}{3} \vspace{1mm} \\ 
 1, 1, 1, 1, 1 \end{matrix} \right]_{n} 
 \big( 74 n^2 + 27 n + 3 \big). 
\end{equation*}
 This leads us to generalize \eqref{EKHADfor2764} by setting 
\begin{equation}\label{inputfor2764}
 \mathcal{F}_{\alpha, \beta_1, \beta_2, \gamma_1, \gamma_2, \gamma_3, \gamma_4}(n, k) 
 = \left[ \begin{matrix} 
 \alpha, 1 - \alpha, (\beta_1 + \beta_2) n + \gamma_1 \vspace{1mm} \\ 
 \beta_1 n + \gamma_2, \beta_2 n + \gamma_3, \gamma_4 
 \end{matrix} \right]_{k + \delta n} 
\end{equation}
 and by applying {\tt EKHAD}-normalization. 
 For example, for the 
 $(\alpha$, 
 $ \beta_1$, 
 $ \beta_2$,
 $ \gamma_1$, 
 $ \gamma_2$, 
 $ \gamma_3$, 
 $ \gamma_4$, 
 $ \delta) $ $ = $ 
 $ \big( \frac{1}{2}$, 
 $ 1$, $ 2$, $ 0$, $ 1$, $ 1$, $ 1$, $ 0 \big)$ case of \eqref{inputfor2764}, 
 we obtain from the WZ relation in 
 \eqref{mainGH} that 
\begin{align*}
 \frac{128}{\pi ^2} 
 = & \sum_{n = 0}^{\infty} 
 \left( \frac{256}{729} \right)^{n} 
 \left[ \begin{matrix} 
 \frac{1}{4}, \frac{1}{4}, \frac{1}{4}, \frac{1}{2}, \frac{1}{2}, \frac{1}{2}, \frac{3}{4}, \frac{3}{4}, \frac{3}{4} \vspace{1mm} \\ 
 1, 1, 1, 1, 1, \frac{4}{3}, \frac{4}{3}, \frac{5}{3} \end{matrix} \right]_{n} \\
 & \times \text{ \footnotesize{ $\big( 7568 n^6+17184 n^5+15256 n^4+6748 n^3+1574 n^2+186 n+9 \big).$}}
\end{align*}

\subsection*{Acknowledgments}
 The author is grateful to acknowledge support from a Killam Fellowship from the Killam Trusts. The author thanks Karl Dilcher and Lin 
 Jiu for a useful discussion concerning this paper. 
 The author is thankful to a number of anonymous Reviewers of this research paper 
 for the detailed commentary provided. 

\subsection*{Data availability statement}
 No supporting data associated with this research paper are available. 

\subsection*{Conflict of interest statement}
 There are no conflicts of interest associated with this research paper.

 \

\noindent John M.\ Campbell

\vspace{0.1in}

\noindent Department of Mathematics and Statistics

\noindent Dalhousie University

\noindent 6299 South St, Halifax, NS B3H 4R2

\vspace{0.1in}

\noindent {\tt jh241966@dal.ca}


\begin{thebibliography}{99}

\bibitem{Au2025}
 K.\ C.\ Au, 
 Wilf-{Z}eilberger seeds and non-trivial hypergeometric identities, 
 \emph{J.\ Symbolic Comput.} {\bf 130} (2025), 
 Paper No.\ 102421. 

\bibitem{Bailey1935}
 W.\ N.\ Bailey, 
 {\it Generalised hypergeometric series}, 
 Cambridge University Press, Cambridge, 1935. 

\bibitem{Bauer1859}
 G.\ Bauer, 
 Von der Coefficienten der Reihen von Kugelfunctionen einer Variabeln, 
 {\it J.\ Reine Angew.\ Math.} {\bf 56} (1859) 101--121. 

\bibitem{Berndt1994ParII}
 B.\ C.\ Berndt, 
 {\it Ramanujan's notebooks. {P}art {II}}, 
 Springer-Verlag, New York, 1989. 

\bibitem{Berndt1994PartIV}
 B.\ C. Berndt, 
 {\it Ramanujan's Notebooks. {P}art {IV}}, 
 Springer-Verlag, New York, 1994. 

\bibitem{Campbell2025}
 J.\ M.\ Campbell, 
 Hypergeometric accelerations with shifted indices, 
 {\it J.\ Difference Equ. Appl.} {\bf 31} (2025) 1169--1192. 

\bibitem{ChenChu2021IJNT}
 X. Chen and W. Chu, 
 {$q$}-analogues of {G}uillera's two series for {$\pi^{\pm 2}$} with convergence rate {$\frac{27}{64}$}, 
 \emph{Int.\ J.\ Number Theory} {\bf 17} (2021), no.\ 1, 
 71--90. 

\bibitem{ChenChu2021RJ}
 X.\ Chen and W. Chu, 
 Hidden {$q$}-analogues of {R}amanujan-like {$\pi$}-series, 
 \emph{Ramanujan J.} {\bf 54} (2021), no. 3, 
 625--648. 

\bibitem{Chu2018}
 W.\ Chu, 
 {$q$}-series reciprocities and further {$\pi$}-formulae, 
 {\it Kodai Math.\ J.} {\bf 41} (2018) 512--530. 

\bibitem{EkhadZeilberger1994}
 S.\ B.\ Ekhad and D. Zeilberger, 
 A {WZ} proof of {R}amanujan's formula for {$\pi$}. 
 In \emph{Geometry, analysis and mechanics}, 
 pages 107--108. World Sci.\ Publ., River Edge, NJ, 1994. 

\bibitem{GasperRahman2004}
 G.\ Gasper and M.\ Rahman, 
 {\it Basic hypergeometric series}, 
 Cambridge University Press, Cambridge, 2004. 

\bibitem{Guillera2006}
 J.\ Guillera, 
 Generators of some {R}amanujan formulas, 
 \emph{Ramanujan J.} {\bf 11} (2006), no. 1, 
 41--48. 

\bibitem{Guillera2011}
 J.\ Guillera, 
 A new {R}amanujan-like series for {$1/\pi^2$}, 
 \emph{Ramanujan J.} {\bf 26} (2011), no. 3, 
 369--374. 

\bibitem{Guillera2007}
 J.\ Guillera, 
 \emph{Series de Ramanujan: Generalizaciones y conjeturas}, 
 PhD Thesis, Universidad de Zaragoza, 2007. 

\bibitem{Guillera2002}
 J.\ Guillera, 
 Some binomial series obtained by the {WZ}-method, 
 {\it Adv.\ in Appl.\ Math.} {\bf 29} (2002) 599--603. 

\bibitem{Guillera2018}
 J.\ Guillera, 
 W{Z} pairs and {$q$}-analogues of {R}amanujan series for {$1/\pi$}, 
 \emph{J.\ Difference Equ. Appl.} {\bf 24} (2018), no. 12, 
 1871--1879. 

\bibitem{Guo2018central}
 V.\ J.\ W.\ Guo, 
 A {$q$}-analogue of a {R}amanujan-type supercongruence involving central binomial coefficients, 
 {\it J.\ Math. Anal. Appl.} {\bf 458} (2018) 590--600. 

\bibitem{Guo2018J2}
 V.\ J.\ W.\ Guo, 
 A {$q$}-analogue of the ({J}.2) supercongruence of {V}an {H}amme, 
 {\it J.\ Math.\ Anal. Appl.} {\bf 466} (2018) 776--788. 

\bibitem{Guo2020}
 V.\ J.\ W.\ Guo, 
 {$q$}-analogues of three Ramanujan-type formulas for {$1/\pi$}, 
 \emph{Ramanujan J.} {\bf 52} (2020), no. 1, 
 123--132. 

\bibitem{GuoLiu2018}
 V. J. W. Guo and J.-C.\ Liu, 
 {$q$}-analogues of two Ramanujan-type formulas for {$1/\pi$}, 
 \emph{J.\ Difference Equ.\ Appl.} {\bf 24} (2018), no. 8, 
 1368--1373. 

\bibitem{GuoZudilin2018}
 V.\ J.\ W.\ Guo and W.\ Zudilin, 
 Ramanujan-type formulae for {$1/\pi$}: {$q$}-analogues, 
 \emph{Integral Transforms Spec.\ Funct.} {\bf 29} (2018), no. 7, 
 505--513. 

\bibitem{PetkovsekWilfZeilberger1996}
 M.\ Petkov\v sek, H.\ S.\ Wilf, D. Zeilberger, 
 {{$A=B$}}, 
 A K Peters, Ltd., Wellesley, MA, 1996. 

\bibitem{Ramanujan1962}
 S. Ramanujan, 
 {\it Collected papers}, 
 Chelsea, New York, 1962. 

\bibitem{Ramanujan1914}
 S.\ Ramanujan, 
 Modular equations and approximations to $\pi$, 
 {\it Quart. J.} {\bf 45} (1914) 350--372. 

\bibitem{Slater1952}
 L.\ J.\ Slater, 
 Further identities of the {R}ogers-{R}amanujan type, 
 {\it Proc. London Math. Soc. (2)} {\bf 54} (1952) 147--167. 

\bibitem{Wei2023}
 C. Wei, 
 On two double series for {$\pi$} and their {$q$}-analogues, 
 \emph{Ramanujan J.} {\bf 60} (2023), no. 3, 
 615--625. 
 
\bibitem{WilfZeilberger1990}
 H.\ S.\ Wilf and D.\ Zeilberger, 
 Rational functions certify combinatorial identities, 
 \emph{J.\ Amer.\ Math.\ Soc.} {\bf 3} (1990), no. 1, 
 147--158. 

\bibitem{Zeilberger1993}
 D.\ Zeilberger, 
 Closed form (pun intended!). 
 In \emph{A tribute to {E}mil {G}rosswald: number theory and related analysis}, 
 pages 579--607. Amer. Math. Soc., Providence, RI, 1993. 
 
\end{thebibliography}
\end{document}